\documentclass[11pt]{article}
\usepackage[utf8]{inputenc}
\usepackage[total={4.9in, 7.7in}]{geometry}
\usepackage{mathtools}
\usepackage{amsmath}
\usepackage{amssymb}
\usepackage{amsthm}
\usepackage{amsfonts}
\usepackage{physics}
\usepackage{mathrsfs}
\usepackage{graphicx}
\usepackage{nopageno}
\usepackage{times}
\usepackage[colorinlistoftodos]{todonotes}
\usepackage[colorlinks=true, allcolors=blue]{hyperref}
\usepackage{enumitem}

\usepackage{fancyhdr}
\fancypagestyle{titlefooter}{
  \fancyhf{}
  
  \fancyfoot[C]{$^*$This research was supported by NSF grant no. 1950563.}
}

\newtheorem{theorem}{Theorem}[section]

\newtheorem{proposition}{Proposition}[section]
\newtheorem{corollary}{Corollary}[proposition]

\newcommand{\nat}{\mathbb{N}}
\newcommand{\zz}{\mathbb{Z}}

\newcommand{\stin}{\subseteq}

\newcommand{\mb}[1]{\mathbf{#1}}

\newcommand{\sett}[1]{\left\{ #1 \right\}}
\newcommand{\set}[1]{\left\{ #1 \right\}}
\newcommand{\paren}[1]{\left( #1 \right) }
\newcommand{\tb}[1]{\textbf{#1}}
\newcommand{\ti}[1]{\textit{#1}}

\newcommand{\R}{\mathbb{R}}
\newcommand{\Q}{\mathbb{Q}}
\newcommand{\N}{\mathbb{N}}
\newcommand{\Z}{\mathbb{Z}}

\newcommand{\frob}[1]{\textnormal{Frob}\left( #1 \right)}

\newcommand{\cond}[1]{\chi\left(#1 \right)}

\begin{document}

\title{\tb{Frobenius templates in certain $\mathbf{2 \times 2}$ matrix rings$\mathbf{^*}$}\vspace{-4ex}}
\author{}
\date{}

\maketitle

\begin{center} \large 
\textbf{\scalebox{1.0909}{Timothy Eller}} \vspace{2mm} \\ 

Georgia Southern University, te02816@georgiasouthern.edu \vspace{4mm}\\

\textbf{\scalebox{1.0909}{Jakub Kraus}} \vspace{2mm}\\

University of Michigan, jakraus@umich.edu \vspace{4mm}

\textbf{\scalebox{1.0909}{Yuki Takahashi}} \vspace{2mm}\\

Grinnell College, takahash@grinnell.edu \vspace{4mm}

\textbf{\scalebox{1.0909}{Zhichun (Joy) Zhang}} \vspace{2mm}\\

Swarthmore College, zzhang3@swarthmore.edu \vspace{4mm}\end{center}

\thispagestyle{titlefooter}

\begin{abstract}
The classical Frobenius problem is to find the largest integer that cannot be written as a linear combination of a given set of positive, coprime integers using nonnegative integer coefficients. Prior work has generalized the classical Frobenius problem from integers to Frobenius problems in other rings. This paper explores Frobenius problems in various rings of (upper) triangular $2 \times 2$ matrices with constant diagonal.
\end{abstract}

\tb{Key words and phrases}: Frobenius template, Frobenius problem, coin problem, monoid, ring.

\tb{AMS Subject Classification}: 11D07, 11B05.

\section{Introduction}
Let $\N$ be the set of nonnegative integers, $\Z$ the set of integers, $\Z^+ \vcentcolon = \N \setminus \set{0}$,
and, for $\alpha_1, \dots, \alpha_n \in \N$, $MN(\alpha_1, \dots, \alpha_n) \vcentcolon= \{\sum_{i=1}^n \lambda_i \alpha_i : \lambda_1, \dots ,\lambda_n \in \mathbb N\}$. 
A \textit{list} in a set $S$ is a finite sequence of elements from $S$. If $n \in \Z^+$ and $(\alpha_1, \dots, \alpha_n)$ is a list in $\N$,
we will say $\alpha_1, \dots, \alpha_n$ are \ti{coprime} to mean $\gcd (\alpha_1, \dots, \alpha_n) = 1$. If we assume $\gcd(\alpha_1, \dots, \alpha_n) \neq 1$ when $\alpha_1 = \dots = \alpha_n = 0$,
then a list in $\N$ is coprime if and only if the list contains positive integers and the positive integers in the list are coprime. For an additive group $G$, if $S \subseteq G$ and $g \in G$, then let $g + S \vcentcolon= \set{g + s : s \in S}$.

Here is a 19th century theorem, due to Sylvester and Frobenius and others, restated to suit our purpose:
\begin{theorem}
   If $\alpha_1, \dots, \alpha_n\in \N$ are coprime, then for some $w\in \mathbb N$, $w+\mathbb N\subseteq MN(\alpha_1, \dots, \alpha_n)$.
\end{theorem}
\noindent Note that because $0\in \mathbb N$, if $w+\mathbb N\subseteq MN(\alpha_1, \dots, \alpha_n)$ then \\$w\in MN(\alpha_1, \dots, \alpha_n)$.
Note also that because $\mathbb N$ is closed under addition, if $w+\mathbb N\subseteq MN(\alpha_1, \dots, \alpha_n)$ and $b\in \mathbb N$, then $w+b+\mathbb N\subseteq w+\mathbb N\subseteq MN(\alpha_1, \dots, \alpha_n)$.
Thus, if we define, for $\alpha_1, \dots, \alpha_n\in \mathbb Z^+$, $\frob{\alpha_1, \dots, \alpha_n} \vcentcolon= \{w\in \mathbb N:w+\mathbb N\subseteq MN(\alpha_1, \dots, \alpha_n)\}$,
we have that either $\frob{\alpha_1, \dots, \alpha_n}=\emptyset$ or \\$\frob{\alpha_1, \dots, \alpha_n}=\chi (\alpha_1, \dots, \alpha_n)+\mathbb N$, where $\chi(\alpha_1, \dots, \alpha_n)$ is the least element of $\frob{\alpha_1, \dots, \alpha_n}$.

The classical Frobenius problem \cite{brauer}, sometimes called the ``coin problem,'' is to evaluate $\chi(\alpha_1, \dots, \alpha_n)$ at coprime lists $(\alpha_1, \dots, \alpha_n)$ in $\N$.
Frobenius used to discuss this problem in his lectures \cite{alfonsin2005diophantine}, so this area of number theory happens to be named after him.
For $n=2$, there is a formula: if $\alpha, \beta \in \N$ and $\gcd (\alpha, \beta)=1$, then $\chi(\alpha, \beta)=(\alpha-1)(\beta-1)$.
For $n>2$ there are no known general formulas, although there are formulas for special cases and there are algorithms; see \cite{trimm}. The classical Frobenius problem inspired the results in \cite{gaussian}, which in turn inspired Nicole Looper \cite{looper} to generalize the classical Frobenius problem by introducing the notion of a Frobenius template (defined below), allowing one to pose similar problems in rings other than $\Z$.

All of our rings will contain a multiplicative identity, 1, and most will be commutative.
An \textit{additive monoid} in a ring $R$, is a subset of $R$ that is closed under addition and contains the identity $0$; monoids, unlike groups, do not have to contain inverses. A \textit{Frobenius template} in a ring $R$ is a triple $(A, C, U)$ such that: \begin{enumerate}[label=(\roman*)]
    \item $A$ is a nonempty subset of $R$,
    \item $C$ and $U$ are additive monoids in $R$, and
    \item for all lists $(\alpha_1, \dots, \alpha_n)$ in $A$, \\
    $MN(\alpha_1, \dots ,\alpha_n)=\{\sum_{i=1}^n \lambda_i\alpha_i:\lambda_1, \dots ,\lambda_n\in C\}$ is a subset of $U$.
\end{enumerate} For any valid template, the properties above guarantee that $MN(\alpha_1, \dots ,\alpha_n)$ will also be an additive monoid in $R$.
The \textit{Frobenius set} of a list $\alpha_1, \dots ,\alpha_n\in A$ is 
$\frob{\alpha_1, \dots, \alpha_n} \vcentcolon= \set{w\in R:w+U\subseteq MN(\alpha_1, \dots ,\alpha_n)}$. Notice that $w \in \frob{\alpha_1, \dots, \alpha_n}$ implies that $w \in MN(\alpha_1, \dots, \alpha_n)$ since $0 \in U$, so $\frob{\alpha_1, \dots, \alpha_n} \subseteq MN(\alpha_1, \dots, \alpha_n) \subseteq U$. Notice also that the assumption that $U$ is closed under addition guarantees that if $w \in \frob{\alpha_1, \dots, \alpha_n}$, then $w + U \subseteq \frob{\alpha_1, \dots, \alpha_n}$.

For a given template $(A, C, U)$, the corresponding \textit{Frobenius problem} is the following pair of tasks:
\begin{enumerate}
    \item Determine for which lists $\alpha_1, \dots ,\alpha_n\in A$ it is true that $\frob{\alpha_1, \dots, \alpha_n}\neq \emptyset$.
    \item For lists $\alpha_1, \dots ,\alpha_n$ such that $\frob{\alpha_1, \dots, \alpha_n}\neq \emptyset$, describe the set $\frob{\alpha_1, \dots, \alpha_n}$.
\end{enumerate}
Note that in all the Frobenius templates that have been studied so far, it has always been the case that $\frob{\alpha_1, \dots, \alpha_n}$, when nonempty, is a finite union of sets of the form $w + U$ for some $w \in R$.

The classical Frobenius problem revolves around the template $(\N, \N, \N)$ in the ring $\Z$. As discussed earlier, if $n = 2$ and $\alpha_1, \alpha_2 \in \N$, then $\frob{\alpha_1, \alpha_2}$ is nonempty if (and only if) $\alpha_1, \alpha_2$ are coprime, and $\frob{\alpha_1, \alpha_2} =\chi(\alpha_1, \alpha_2)+ \N = (\alpha_1 - 1)(\alpha_2 - 1) + \N$ in such cases. So the classical Frobenius problem has been completely solved when $n = 2$.

Our definition above of a Frobenius template is slightly less general than Looper's definition in \cite{looper}. The difference is that in Looper's templates, $U=U(\alpha_1, \dots ,\alpha_n)$ is allowed to vary with the list $\alpha_1, \dots ,\alpha_n$.
This is absolutely necessary in order to get interesting and nontrivial results when the ring involved is a subring of $\mathbb C$, the complex numbers, that is not contained in $\mathbb R$, the real numbers.
The ring of Gaussian integers, $\mathbb Z[i]=\{a+bi:a, b\in \mathbb Z\}$, is such a ring, a particularly famous member of the family of rings $\{\mathbb Z[\sqrt{m}i] : m\in \mathbb Z^+ \}$.

To see why we must let the element $U$ of a Frobenius template in such a ring depend on the list $\alpha_1, \dots ,\alpha_n\in A$, here is an example in $\mathbb Z[\sqrt{3}i]$.
Suppose $C=\{a+b\sqrt{3}i:a, b\in \mathbb N\} = \set{re^{i\theta} \in \Z[\sqrt{3}i] : 0 \leq r \text{ and } 0 \leq \theta \leq \frac{\pi}{2}}$ and $A=C \setminus \{0\}$.
Then consider $\alpha= 2e^{i \frac{\pi}{3}} = 1+\sqrt{3}i \in A$ and $\beta = 1 \in A$.
What should $U$ be for this list?
We set $U(\alpha, \beta)=\{re^{i\theta}\in \mathbb Z[\sqrt{3}i] : 0 \leq r \text{ and } 0\leq \theta \leq \frac{5\pi}{6}\}$, an angular sector in $\mathbb Z[\sqrt{3}i]$.
We so choose $U(\alpha, \beta)$ because $MN(\alpha, \beta)=\set{\lambda_1 (2e^{i \frac{\pi}{3}}) +\lambda_2 : \lambda_1, \lambda_2 \in C} \subseteq U(\alpha, \beta)$,
and no angular sector in $\mathbb Z[\sqrt{3}i]$ properly contained in $U(\alpha,\beta)$ contains $MN(\alpha, \beta)$.
Further, if $\hat{U}$ is an additive monoid in $\mathbb Z[\sqrt{3}i]$ which properly contains $U(\alpha, \beta)$, then no translate $w+\hat{U}$, $w\in \mathbb C$, is contained in $MN(\alpha, \beta)$. Thus $U(\alpha, \beta)$ is the only possible additive monoid in $\mathbb Z[\sqrt{3}i]$ containing $MN(\alpha, \beta)$ such that $\frob{\alpha, \beta} = \{w:w+U(\alpha, \beta)\subseteq MN(\alpha, \beta)\}$ might possibly be nonempty.

We leave to the reader the verification of the claims in the previous paragraph.
We hope that the main point is clear: 
letting $U$ vary with the list $\alpha_1, \dots ,\alpha_n\in A\subseteq \mathbb Z[\sqrt{m}i]$ in the formulation of Frobenius problems in $\mathbb Z[\sqrt{m}i]$ is necessitated by the nature of multiplication in the complex numbers.
For an appreciation of Frobenius problems in such settings, see \cite{gaussian} and \cite{splitprimes}.

We do not take on similar difficulties here.
In the next section we give some more or less obvious results in various Frobenius templates, then review previous results concerning templates in the rings $\mathbb Z[\sqrt{m}]$ with $m\in \mathbb Z^+ \setminus  \{n^2:n\in \mathbb Z^+\}$. In the third section, we classify the Frobenius set in a pleasing modification of the classical Frobenius template. In the fourth section, we solve Frobenius problems in rings of $2 \times 2$ (upper) triangular matrices with constant diagonal and entries from a ring $Q$, where different choices of $Q$ allow different templates. In the last section, we further generalize the idea of a Frobenius template and explore an example of this generalization.

\section{Some fundamentals and known results}

Throughout this section, $R$ will be a ring with multiplicative identity 1. A list $(\alpha_1, \dots ,\alpha_n)$ in $R$ \textit{spans unity} in $R$ if and only if $1=\lambda_1\alpha_1+...+\lambda_n\alpha_n$ for some $\lambda_1, \dots ,\lambda_n\in R$.

\begin{proposition}
Let $(A, C, U)$ be a Frobenius template in $R$ such that $1 \in U$.
If $(\alpha_1, \dots ,\alpha_n)$ is a list in $A$ such that $\frob{\alpha_1, \dots, \alpha_n}\neq \emptyset$, then $(\alpha_1, \dots ,\alpha_n)$ spans unity in $R$.
\end{proposition}
\begin{proof}
Let $w \in \frob{\alpha_1, \dots, \alpha_n}$, i.e. $w\in R$ and $w+U \subseteq MN(\alpha_1, \dots ,\alpha_n)$.
Since $0 \in U$ and $1 \in U$, it follows that $w, w+1 \in MN(\alpha_1, \dots ,\alpha_n)$.
Therefore, for some $\lambda_1, \dots ,\lambda,\gamma_1, \dots ,\gamma_n\in C$, 
$w=\sum_{i=1}^n \lambda_i \alpha_i$ and
$w+1=\sum_{i=1}^n \gamma_i\alpha_i$, so $1=\sum_{i=1}^n (\gamma_i - \lambda_i) \alpha_i$.
\end{proof}

As a warm-up, here are some valid Frobenius templates $(A, C, U)$ with, in most cases, trivial solutions to the corresponding Frobenius problems. In each example, $(\alpha_1, \dots, \alpha_n)$ denotes a list in $A$. Check that $A$ is nonempty, $C$ and $U$ are additive monoids in $R$, and $U$ always contains $MN(\alpha_1, \dots, \alpha_n)$.
\begin{enumerate}
    \item In the template $(R, \set{0}, \set{0})$, we have $MN(\alpha_1, \dots, \alpha_n) = \set{0}$ and\\ $\frob{\alpha_1, \dots, \alpha_n} = \set{0}$. This result extends to any template $(A, C, U)$ with $C = U = \set{0}$.
    
    \item  Similar to example 1, templates of the form $(\set{0}, C, \set{0})$ always have $MN(0) = \set{0}$ and $\frob{0} = \set{0}$.
    
    \item In the template $(R, R, R)$, Proposition 2.1 shows that $(\alpha_1, \dots, \alpha_n)$ spans unity in $R$ if $\frob{\alpha_1, \dots, \alpha_n} \neq \emptyset$. Conversely, if $(\alpha_1, \dots, \alpha_n)$ spans unity, then $\frob{\alpha_1, \dots, \alpha_n} = R$.
    
    \item Suppose that $I$ is an ideal of $R$ and the template is $(\{1\}, I, I)$.
    Clearly $MN(1) = I = U$, so $U \subseteq \frob{1}$. Recall that $\frob{\alpha_1, \dots, \alpha_n}$ is a subset of $MN(\alpha_1, \dots, \alpha_n)$; thus $\frob{1} = U$.
    
    \item Again, let $I$ be an ideal of $R$, and now consider the template $(I, R, I)$.
    When $I = R$, this is example 3.
    In contrast, let $I$ be a proper ideal. Then no list $(\alpha_1, \dots, \alpha_n)$ spans unity in $R$ --- but this is precisely because $1\notin I$, so no facile conclusion based on Proposition 2.1 presents itself.
    Frobenius problems for this class of templates could be interesting.
    For instance, if the template is $(2 \Z, \Z, 2 \Z)$ in $\mathbb Z$, where $2\mathbb Z$ is the ideal of even integers, it is straightforward to prove that $\frob{\alpha_1, \dots, \alpha_n}\neq \emptyset$ if and only if the integers $\abs{\frac{\alpha_1}{2}}, \dots , \abs{\frac{\alpha_n}{2}}$ are coprime.
    In this case, $MN(\alpha_1, \dots ,\alpha_n)=2\mathbb Z$ and $\frob{\alpha_1, \dots, \alpha_n} = 2 \text{Frob}'\paren{\frac {\alpha_1}{2}, \dots, \frac {\alpha_n}{2}}$, where $\text{Frob}'$ denotes Frobenius sets with respect to the classical template $(\N, \N, \N)$.
    
    \item If the template is $((0, \infty), [0, \infty), [0, \infty))$ in $\R$, then $MN(\alpha_1, \dots, \alpha_n) = [0, \infty) = \frob{\alpha_1, \dots, \alpha_n}$ for every list $(\alpha_1, \dots, \alpha_n)$. On the other hand, if we restrict the coefficients to the set of rational numbers $\Q$ with the template $((0, \infty), \Q \cap [0, \infty), [0, \infty))$ in $\R$, then $\frob{\alpha_1, \dots, \alpha_n}$ is always empty, as $MN(\alpha_1, \dots, \alpha_n)$ is countable and any translate of $[0, \infty)$ is uncountable.
    
    \item Suppose that $R_1$ and $R_2$ are rings with associated Frobenius templates $(A_1, C_1, U_1)$ and $(A_2, C_2, U_2)$,  respectively.
    Let $A_1 \times A_2$ be the Cartesian product of $A_1$ and $A_2$. The same goes for $C_1 \times C_2$ and $U_1 \times U_2$, 
    which are additive monoids (under componentwise addition) in the product ring $R_1 \times R_2$. It is simple to see that $(A_1 \times A_2, C_1 \times C_2, U_1 \times U_2)$ is a valid Frobenius template in $R_1 \times R_2$. Let $(\beta_1, \dots ,\beta_n)$ be a list in $A_1 \times A_2$ with each $\beta_i=(\gamma_i, \mu_i)$ for $i=1, \dots, n$. 
    If $\frob{\beta_1, \dots, \beta_n} \neq \emptyset$ in $(A_1 \times A_2, C_1 \times C_2, U_1 \times U_2)$, then\\ $\frob{\beta_1, \dots ,\beta_n} = \textnormal{Frob}_1(\gamma_1, \dots, \gamma_n) \times \textnormal{Frob}_2(\mu, \dots, \mu_2)$, where $\textnormal{Frob}_1$ and $\textnormal{Frob}_2$ are Frobenius sets in $(A_1, C_1, U_1)$ and $(A_2, C_2, U_2)$, respectively. For $m \in \Z^+$, this result extends to the product ring $R_1 \times \dots \times R_m$.
\end{enumerate}

Besides the Gaussian integers \cite{gaussian, splitprimes}, prior work on generalized Frobenius problems has concentrated on templates in the real subring $\mathbb Z[\sqrt{m}] \vcentcolon= \{a+b\sqrt{m}:a, b\in \mathbb Z\}$, where $m \in \Z^+$ is not a perfect square.
Below, we showcase some highlights from this work.
If $m$ is a positive integer with irrational square root, then $\Z[\sqrt{m}]$ is a subring of $\R$.
Let $\N[\sqrt{m}] \vcentcolon=\{a+b\sqrt{m} : a, b\in \N\}$ and $\Z[\sqrt{m}]^+ \vcentcolon=\Z[\sqrt{m}]\cap [0, \infty)$, which are both additive monoids in $\Z[\sqrt{m}]$ that contain $1$.

\begin{enumerate}
    \item In the template $(\mathbb N[\sqrt{m}], \mathbb N[\sqrt{m}], \mathbb N[\sqrt{m}])$, the result below is proven in \cite{looper}:
    
    {\theorem If $\alpha_1, \dots ,\alpha_n \in \mathbb N[\sqrt{m}]$, $\alpha_i=a_i+b_i\sqrt{m}$, $a_i, b_i\in \mathbb N$, $i=1, \dots ,n$, 
    then $\frob{\alpha_1, \dots, \alpha_n}\neq \emptyset$ if and only if $(\alpha_1, \dots, \alpha_n)$ spans unity in $\mathbb Z[\sqrt{m}]$ and at least one of $a_1, \dots ,a_n$, $b_1, \dots ,b_n$ is zero.}
    
    The result stated in \cite{looper}, Theorem 3, is slightly weaker than the formulation above, but needlessly so, since the proof in \cite{looper} proves the statement given here.

    \item In the template $(\mathbb Z[\sqrt{m}]^+, \mathbb Z[\sqrt{m}]^+, \mathbb Z[\sqrt{m}]^+)$, the following is nearly proven in \cite{beneish}: 
    
    {\theorem
    If $\alpha_1, \dots ,\alpha_n\in \mathbb Z[\sqrt{m}]^+$, then the following are equivalent:
    \begin{enumerate}[label=(\roman*)]
        \item $\frob{\alpha_1, \dots, \alpha_n}\neq \emptyset$;
        \item $\frob{\alpha_1, \dots, \alpha_n}=\mathbb Z[\sqrt{m}]^+$;
        \item $(\alpha_1, \dots ,\alpha_n)$ spans unity in $\mathbb Z[\sqrt{m}]$.
    \end{enumerate}
    }
    Since Theorem 2 of \cite{beneish} proves $(iii) \implies (ii)$, we simply appeal to Proposition 2.1 to prove the statement above, which is somewhat stronger than Theorem 2 in \cite{beneish}.
    
    Also in \cite{beneish}, with reference to the template $(\mathbb N[\sqrt{m}], \mathbb N[\sqrt{m}], \mathbb N[\sqrt{m}])$,\\ $\frob{\alpha_1, \dots, \alpha_n}$ is determined in a very special class of cases. Recall that for coprime positive integers $c_1, \dots, c_n$, $\cond{c_1, \dots, c_n}$ is the smallest $w \in \Z^+$ such that $w + \N \subseteq \set{\sum_{i=1}^n \lambda_i c_i : \lambda_1, \dots ,\lambda_n\in \N }$.
    
    {\theorem
    For $a_1, \dots ,a_r, b_1, \dots ,b_s\in \N$, \\ $(a_1, \dots ,a_r, b_1\sqrt{m}, \dots ,b_s\sqrt{m})$ is a list in $\N[\sqrt{m}]$.
    We have
    \[
    \begin{split}
       &\frob{a_1, \dots ,a_r, b_1\sqrt{m}, \dots ,b_s\sqrt{m}} \neq \emptyset\\ &\iff a_1, \dots ,a_r, b_1m, \dots ,b_sm \text{ are coprime,} 
    \end{split}
    \] 
    and, in such cases,
    \[
    \begin{split}
    &\frob{a_1, \dots ,a_r, b_1\sqrt{m}, \dots ,b_s\sqrt{m}}\\
        &= \chi(a_1, \dots ,a_r, b_1m, \dots ,b_sm) + \chi(a_1, \dots ,a_r, b_1, \dots ,b_s)\sqrt{m} + \mathbb N[\sqrt{m}] \text{.}
    \end{split}
    \] } Note that the case $s=0$ is allowed in this result. Also, if the integers $a_1, \dots ,a_r, b_1m, \dots ,b_sm$ are coprime, then so are $a_1, \dots ,a_r, b_1, \dots ,b_s$, so the second part of Theorem 2.3 uses well-defined expressions.
    
    \item With $\chi(a_1, \dots, a_n)$ as above, recall that in the $n = 2$ case of the classical template $(\N, \N, \N)$ in $\Z$, $\chi(a_1, a_2) = (a_1 - 1)(a_2 - 1)$ for coprime $a_1, a_2$. In a tour de force in \cite{doyon}, Kim proves a similar formula for the template $(\mathbb N[\sqrt{m}], \mathbb N[\sqrt{m}], \mathbb N[\sqrt{m}])$ in $\mathbb Z[\sqrt{m}]$:
    {\theorem
    Suppose $\alpha=a+b\sqrt{m}, \beta=c+d\sqrt{m}$, $a, b, c, d\in \mathbb N$, $a+b, c+d>0$, $abcd=0$ 
    (see Theorem 2.1, above), and suppose that $\alpha, \beta$ span unity in $\mathbb Z[\sqrt{m}]$.
    Then $\frob{\alpha, \beta}=(\alpha-1)(\beta-1)(1+\sqrt{m})+\mathbb N[\sqrt{m}]$.}
\end{enumerate}

What's next? We propose these categories of rings where one might discover results of interest of the Frobenius type.
\begin{enumerate}
    \item \textit{Subrings of algebraic extensions of }$\Q$.
    
    Prior work on the Gaussian integers and the rings $\mathbb Z[\sqrt{m}]$ has put us into the foothills of a mountain range in this area.
    Besides the extensions of $\Q$ of finite degree, what about the ring of algebraic integers?
    Or, a better choice to start with, the ring of real algebraic integers?
    
    \item \textit{Polynomial rings.}
    
    Somebody we know is working on this, but we haven't heard from him for a year or so.
    
    \item \textit{Rings of square matrices.}
    
    Did we say our rings have to be commutative?
    No, we did not. Keep in mind that in noncommutative rings, the precise definition of\\ $MN(\alpha_1, \dots, \alpha_n)$ becomes important, since for $\alpha_1, \dots, \alpha_n$ from $A$ and $\lambda_1, \dots, \lambda_n$ from $C$, the linear combinations $\sum_{i = 1}^n \lambda_i \alpha_i$ and $\sum_{i = 1}^n \alpha_i \lambda_i$ are not necessarily equal.
\end{enumerate}

\section{Modifying the classical template}

The classical template $(\N, \N, \N)$ uses nonnegative integer coefficients. What happens when we modify the available coefficients? Consider the template $(\N, (n + \N) \cup \set{0}, \N)$ for some $n \in \Z^+$. When $n = 1$, this is the classical template.

Remember that when $a_1, \dots, a_k \in \N$ are coprime, $\cond{a_1, \dots, a_k}$ is the unique integer such that $\frob{a_1 \dots, a_k} = \cond{a_1, \dots, a_k} + \N$ with respect to the classical template. In the results that follow, $\cond{a_1, \dots, a_k}$ retains this meaning, whereas $\text{Frob}$ and $MN$ reference the modified template $(\N, (n + \N) \cup \set{0}, \N)$.

{\proposition If $a_1, \dots, a_k \in \N$ are coprime, then
\[ (a_1 + \dots + a_k)n + \cond{a_1, \dots, a_k} + \N \subseteq \frob{a_1, \dots, a_k}\text. \]}

\begin{proof}
Suppose $\omega \in \N$ and $\omega \geq (a_1 + \dots + a_k)n + \cond{a_1, \dots, a_k}$. We will show $\omega \in \frob{a_1, \dots, a_k}$, i.e. 
\[ \omega + \N \subseteq MN(a_1, \dots, a_k) = \sett{ \sum_{i=1}^k \lambda_i a_i : \lambda_1, \dots, \lambda_k \in (n + \N) \cup \set{0} }\text. \]
Suppose $f \in \N$, $f \geq \omega$. It suffices to show that $f \in MN(a_1, \dots, a_k)$. Since $f \geq \omega$, we have that $f - (a_1 + \dots + a_k)n \geq \cond{a_1, \dots, a_k}$, so there exist coefficients $\gamma_1, \dots, \gamma_k \in \N$ such that $f = (a_1 + \dots + a_k)n + \sum_{i=1}^k \gamma_i a_i = \sum_{i=1}^k (\gamma_i + n)a_i$. Taking $\lambda_i = \gamma_i + n \geq n$ for $i =1, \dots, k$ yields the desired result.
\end{proof}

While reading the next proof, remember that $\cond{a, b} = (a - 1)(b - 1)$ for coprime $a, b \in \N$, so $\cond{a, b} - 1 = ab - a - b$.

{\proposition Let $a, b \in \N$ and $n \in \Z^+$. If $a, b$ are coprime and $n - 1$ is not divisible by $a$ or $b$, then $\frob{a, b} = (a + b)n + \cond{a, b} + \N$. }

\begin{proof}
By the preceding proposition, it suffices to show that $\frob{a, b} \subseteq (a + b)n + \cond{a, b} + \N$. So let $t \in \N \setminus \paren{(a + b)n + \cond{a, b} + \N}$, and suppose that $t \in \frob{a, b}$. We will find a contradiction, which will complete the proof.

Since $t \in \frob{a, b}$, we have $t + \N \subseteq MN(a, b)$, so $an + bn + \cond{a, b} - 1 \geq t$ is an element of $MN(a, b)$, so there exist $\lambda_1, \lambda_2 \in (n + \N) \cup \set{0}$ such that
\[ \cond{a, b} - 1 = \lambda_1 a + \lambda_2 b - an - bn = (\lambda_1 - n) a + (\lambda_2 - n) b \text. \]
Now, if both $\lambda_1$ and $\lambda_2$ are $\geq n$, then $\cond{a, b} - 1$ equals a linear combination of $a$ and $b$ with nonnegative integer coefficients, which is impossible. We also cannot have $\lambda_1 = 0 = \lambda_2$ since $\cond{a, b} + an + bn > 1$. So consider the case that $\lambda_1 = 0$ and $\lambda_2 \geq n$. In that case, $(\lambda_2 - n)b = \cond{a, b} - 1 + an = ab - b + a(n - 1)$, so $(\lambda_2 - n - a + 1)b = a(n - 1)$, so $b$ divides $a(n - 1)$. But $a$ and $b$ are coprime, so this implies that $b$ divides $n - 1$, which we have assumed to be false. Similarly $\lambda_1 \geq n$ and $\lambda_2 = 0$ would imply that $(\lambda_1 - n)a = \cond{a, b} - 1 + bn = ab - a + b(n - 1)$, yielding the contradiction that $a$ divides $n - 1$. Thus, all cases yield a contradiction, so our supposition that $t \in \frob{a, b}$ must be false.
\end{proof}

\section{\texorpdfstring{$\mathbf{2 \times 2}$}{Lg} triangular matrices with constant diagonal}

For a ring $Q$ with multiplicative identity $1$, let $R$ be the set $Q^2 = Q \times Q$ under coordinatewise addition, with multiplication in $R$ defined by $(a, b) \cdot (c, d) = (ac, ad + bc)$. $R$ is a ring with multiplicative identity $(1, 0)$, isomorphic to the ring of upper triangular $2 \times 2$ matrices over $Q$ with constant diagonal: $(a, b) \sim \begin{bmatrix} a & b \\ 0 & a \end{bmatrix}$. If $Q$ is commutative, then $R$ is commutative.

We will concentrate on the cases when $Q$ is a subring of the real field $\R$ containing $1$, and the Frobenius template is 
\[ 
\begin{split}
    &\paren{ A(Q), C(Q), U(Q) } \\
    &= \paren{ \paren{Q \cap (0, \infty)} \times \paren{Q \cap [0, \infty)}, \paren{Q \cap [0, \infty)}^2, \paren{Q \cap [0, \infty)}^2}.
\end{split} \]

{\proposition Let $Q$ be a subring of $\R$ containing $1$. For any positive integer $n$, let $(\alpha_1, \dots, \alpha_n)$ be a list in $A(Q) \setminus \paren{ Q \times \set{ 0 } }$. Then $\frob{\alpha_1, \dots, \alpha_n} = \emptyset$.}

\begin{proof}
Let $\alpha_i = (a_i, b_i)$, $i = 1, \dots, n$. Elements of $MN(\alpha_1, \dots, \alpha_n)$ look like
\[
\begin{split}
    (t, u) 
    &= \sum_{i=1}^n (a_i,b_i) \cdot (c_i,d_i) \\
    &= \sum_{i=1}^n (a_ic_i, a_id_i + b_ic_i) \\
    &=  \paren{\sum_{i=1}^n a_ic_i, \sum_{i=1}^n \paren{ a_id_i + b_ic_i} },
\end{split}
\]
where $c_i \geq 0$ and $d_i \geq 0$ for all $i = 1, \dots, n$.

Suppose that $\frob{\alpha_1, \dots, \alpha_n}$ is nonempty, so there is a tuple $(t, u) \in Q^2$ such that $(t, u) + (Q \cap [0, \infty))^2 \subseteq MN(\alpha_1, \dots, \alpha_n)$. Therefore, for all $q \geq 0 $, we have $(t, u) + (q, 0) = (t + q, u) \in MN(\alpha_1, \dots, \alpha_n)$. So for larger and larger $q$, $t + q = \sum_{i=1}^n a_i c_i(q)$ and $u = \sum_{i = 1}^n ( a_i d_i(q) + b_i c_i(q) )$ for some nonnegative integers $c_i(q)$ and $d_i(q)$ ($i = 1, \dots, n$) that vary with $q$. Clearly $ \sum_{i = 1}^n a_i c_i(q) = t + q \to \infty$ as $q \to \infty$. Because each fixed $a_i, b_i > 0$, and each varying $c_i(q), d_i(q) \geq 0$, it follows that $\max_{1 \leq i \leq n} c_i(q) \to \infty$ as $q \to \infty$, and consequently $u = \sum_{i = 1}^n (a_i d_i(q) + b_i c_i(q)) \to \infty$ as $q \to \infty$. But $u$ is a given constant; it does not vary with $q$. Therefore, our supposition must be false, so $\frob{\alpha_1, \dots, \alpha_n} = \emptyset$.
\end{proof}

Moving on, we now know a sufficient condition for $\frob{\alpha_1, \dots, \alpha_n} = \emptyset$ in a large swath of templates. The natural question is whether $\frob{\alpha_1, \dots, \alpha_n}$ is ever nonempty, and the answer is yes.

{\proposition Let $Q$ be a subfield of  $\R$. For any positive integer $n$, let\\ $(\alpha_1, \dots, \alpha_n)$ be a sequence of $2$-tuples $\alpha_i = (a_i, b_i) \in A(Q)$ (for $i = 1, \dots, n$) satisfying $b_1 = 0$. Then $\frob{\alpha_1, \dots, \alpha_n} = MN(\alpha_1,\dots, \alpha_n) = U(Q)$.}

\begin{proof}
We have 
\[ MN(\alpha_1, \dots, \alpha_n) = \set{ \paren{ \sum_{i = 1}^n a_i c_i, \sum_{i = 1}^n a_i d_i + \sum_{i = 2}^n b_i c_i } : c_i, d_i \in Q \cap [0, \infty) } \text. \] 
It will suffice to show that $(0, 0) \in \frob{\alpha_1, \dots, \alpha_n}$. Let $f, g \geq 0$, $f, g \in Q$, so that $f, g \in (0, 0) + U(Q)$. Then use the coefficients $c_1 = \frac{f}{a_1}$, $d_1 = \frac{g}{a_1}$, and $c_j = d_j = 0$ for $j = 2, \dots, n$ to show that $(f, g) \in MN(\alpha_1, \dots, \alpha_n)$, so $(0, 0) \in \frob{\alpha_1, \dots, \alpha_n}$.
\end{proof}

Proposition 4.1 can be generalized, \ti{mutatis mutandis,} to $Q$ being any linearly ordered commutative ring with unity. The same holds for Proposition 4.2, with the additional stipulation that $a_1$ is a unit. In both cases, the ordering is compatible with addition and multiplication.

In other words, setting $Q$ as a subfield of $\R$ yields a boring Frobenius template, since we can shrink elements of $\R$ using coefficients in $c_i, d_i \in Q \cap (0, 1)$. If we prohibit such shrinking, the template becomes much more interesting. Notice that this modification is similar to Section 3's modification of the classical template.

To make the following result fit on the page, we will temporarily adopt the convention that for a tuple $(a, b)$, the notation $(a, b)_+$ stands for $(a, b) + [0, \infty)^2$.

{\proposition Let $Q = \R$, and change $C(\R)$ to be $C = \set{(0, 0)} \cup [1, \infty)^2$. Let $\alpha_1 = (a_1, 0), \alpha_2 = (a_2, b_2)$ form a list $(\alpha_1, \alpha_2)$ of tuples in $A$. If $b_2 = 0$, then $\frob{\alpha_1, \alpha_2} = (\min(a_1, a_2), \min(a_1, a_2))_+$. If $b_2 > 0$, then 
\[ 
\begin{split}
&\frob{\alpha_1, \alpha_2} 
= \\
&\begin{cases}
(a_1, a_1)_+ & a_1 \leq a_2 \text{ and } a_1 \leq b_2 \\
\paren{a_1, a_1}_+ \cup \paren{\frac{a_1 a_2}{b_2} + a_1, b_2}_+ & b_2 < a_1 \leq a_2 \\
\paren{a_1, a_2}_+ \cup \paren{ a_2, \frac{a_1 b_2}{a_2} + a_2}_+ & a_1 > a_2 \text{ and } a_2 \leq b_2 \\
\paren{a_1, a_2}_+ \cup \paren{ a_2, \frac{a_1 b_2}{a_2} + a_2}_+ \cup \paren{ \frac{(a_2)^2}{b_2} + a_1, b_2}_+ & b_2 < a_2 < a_1
\end{cases}
\end{split}
\] }

Simple modifications to the proof of Proposition 4.1 will prove that Proposition 4.1 holds in this template. Since $\frob{\alpha_1, \alpha_2} \subseteq \frob{\alpha_1, \dots, \alpha_n}$, Proposition 4.3 combines with Proposition 4.1 to prove that if $\alpha_i = (a_i, b_i)  \in (0, \infty) \times [0, \infty)$ for $i = 1, \dots, n$ and $n > 1$, then $\frob{\alpha_1, \dots, \alpha_n}$ is nonempty if and only if some $b_i = 0$.

\begin{proof}
\tb{Case 0}: Suppose $b_2 = 0$, so elements of $MN(\alpha_1, \alpha_2)$ look like $(a_1 c_1 + a_2 c_2, a_1 d_1 + a_2 d_2)$ for $c_1, c_2, d_1, d_2 \in [0, \infty) \setminus (0, 1)$. Notice that $f \in (0, \min(a_1,$\\$ a_2))$ or $g \in (0, \min(a_1, a_2))$ implies $(f, g) \notin MN(\alpha_1, \alpha_2)$. Therefore, if $(t, u) \in \frob{\alpha_1, \alpha_2}$, then $t, u \geq \min(a_1, a_2)$. The converse is trivial, so this case is done.

The set $MN(\alpha_1, \alpha_2)$ is 
\[ \sett{ \paren{a_1 c_1 + a_2 c_2, a_1 d_1 + a_2 d_2 + b_2 c_2} : c_1, c_2, d_1, d_2 \in [0, \infty) \setminus (0, 1) } \text{.} \]
The trickiness of the remaining cases is that both coordinates share the coefficient $c_2$. Similar to the last paragraph, notice that $f \in (0, \min(a_1, a_2))$ or $g \in (0, \min(a_1, a_2, b_2))$ implies $(f, g) \notin MN(\alpha_1, \alpha_2)$; therefore, $(t, u) \in \frob{\alpha_1, \alpha_2}$ implies $t \geq \min(a_1, a_2)$ and $u \geq \min(a_1, a_2, b_2)$.

\tb{Case 1}: $a_1 \leq a_2$ and $a_1 \leq b_2$. By remarks above, in this case $\frob{\alpha_1, \alpha_2} \subseteq (a_1, a_1) \cup [0, \infty)^2$. On the other hand, if $f, g \geq a_1$, then $(f, g) = (\frac f {a_1}, \frac g {a_1}) \alpha_1 + (0, 0) \alpha_2$, so $(f, g) \in MN(\alpha_1, \alpha_2)$. Therefore, $(a_1, a_1) + [0, \infty)^2 \subseteq \frob{\alpha_1, \alpha_2}$, so $\frob{\alpha_1, \alpha_2} = (a_1, a_1) + [0, \infty)^2$.

\tb{Case 2}: $0 < b_2 < a_1 \leq a_2$. Then $\frob{\alpha_1, \alpha_2} \subseteq (a_1, b_2) + [0, \infty)^2$ and, by the proof in Case 1,
\[ (a_1, a_1) + [0, \infty)^2 = [a_1, \infty)^2 \subseteq \frob{\alpha_1, \alpha_2} \text. \]
Therefore, to determine $\frob{\alpha_1, \alpha_2}$ in this case, it is sufficient to determine which $(t, u) \in [a_1, \infty) \times [b_2, a_1)$ are in $\frob{\alpha_1, \alpha_2}$.

If $(t, u) \in \frob{\alpha_1, \alpha_2} \subseteq MN(\alpha_1, \alpha_2)$, then for some $c_1, c_2, d_1, d_2 \in \set{0} \cup [1, \infty)$, $t = a_1 c_1 + a_2 c_2$ and $u = a_1 d_1 + a_2 d_2 + c_2 b_2$. Then $b_2 \leq u < a_1 \leq a_2$ implies that $d_1 = d_2 = 0$ and $c_2 = \frac u {b_2}$. Then we have that $t = a_1 c_1 + a_2 \frac u { b_2}$. If $t < a_1 + a_2 \frac u {b_2}$, then $c_1 = 0$ and $t = a_2 \frac u {b_2}$. But this would imply that for any $t'$ such that $t < t' < a_1 + a_2 \frac u {b_2}$, $(t', u) \notin MN(\alpha_1, \alpha_2)$, and this would contradict the assumption that $(t, u) \in \frob{\alpha_1, \alpha_2}$.

Therefore, if $(t, u) \in \frob{\alpha_1, \alpha_2}$ and $b_2 \leq u < a_1$, then $t \geq a_1 + a_2 \frac u {b_2}$. But $(t, u) \in \frob{\alpha_1, \alpha_2}$ implies that $(t, u') \in \frob{\alpha_1, \alpha_2}$ for every $u' \geq u$. Therefore, $t \geq a_1 + a_2 \frac {u'} {b_2}$ for every $u'$ satisfying $u \leq u' < a_1$. Therefore, $t \geq a_1 + \frac{a_1 a_2}{b_2}$.

Thus \begin{align*}
    (a_1, a_1) + [0, \infty)^2 &\subseteq \frob{\alpha_1, \alpha_2} \\
    &\subseteq \paren{(a_1, a_1) + [0, \infty)^2} \cup \paren{ \bigg[\frac{a_1 a_2}{b_2} + a_1, \infty \bigg) \times [b_2, a_1)} \\
    &\subseteq \paren{(a_1, a_1) + [0, \infty)^2} \cup \paren{\paren{\frac{a_1 a_2}{b_2} + a_1, b_2} + [0, \infty)^2 } \text,
\end{align*}
so the proof in this case will be over if we show that $\paren{\frac {a_1 a_2}{b_2} + a_1, b_2}$ is in $\frob{\alpha_1, \alpha_2}$. Suppose that $f \geq \frac {a_1 a_2}{b_2} + a_1$ and $g \geq b_2$. We will see that $(f, g) \in MN(\alpha_1, \alpha_2)$. We may as well assume that $g < a_1$. Then $(f, g) = \paren{\frac 1 {a_1} \paren{f - \frac{g a_2}{b_2}}, 0} \alpha_1 + \paren{ \frac g {b_2}, 0} \alpha_2$, so $(f, g) \in MN(\alpha_1, \alpha_2)$ because $\frac g {b_2} \geq 1$ and $\frac 1 {a_1}\paren{f - \frac{g a_2}{b_2}} \geq \frac{1}{a_1} \paren{\frac{a_1 a_2}{b_2} + a_1 - \frac{a_1 a_2}{b_2} } = 1$.

\tb{Case 3}: $a_2 < a_1$ and $a_2 \leq b_2$. By arguments on display above, $(a_1, a_1) + [0, \infty)^2 \subseteq \frob{\alpha_1, \alpha_2} \subseteq (a_2, a_2) + [0, \infty)^2$. But it is also easy to see that $(a_1, a_2) \in \frob{\alpha_1, \alpha_2}$ ($\implies (a_1, a_2) + [0, \infty)^2 \in \frob{\alpha_1, \alpha_2}$): if $f \geq a_1$ and $g \geq a_2$, then $(f, g) = \paren{\frac f {a_1}, 0} \alpha_1 + \paren{0, \frac g {a_2}} \alpha_2 \in MN(\alpha_1, \alpha_2)$.

It remains to determine $\frob{\alpha_1, \alpha_2} \cap \paren{[a_2, a_1) \times [a_2, \infty) }$. Suppose that $(t, u) \in \frob{\alpha_1, \alpha_2}$ and $a_2 \leq t < a_1$. Let $c_1, c_2, d_1, d_2 \in \set{0} \cup [1, \infty)$ satisfy $t = a_1 c_1 + a_2 c_2$, $u = a_1 d_1 + a_2 d_2 + c_2 b_2$. Then $a_2 \leq t < a_1$ implies that $c_1 = 0$ and $c_2 = \frac t {a_2}$. If $u < a_2 + \frac{t b_2}{a_2}$ then $d_1 = d_2 = 0$ and $u = \frac{tb_2}{a_2}$. But then if $u < u' < a_2 + \frac {t b_2}{a_2}$, $(t, u') \notin MN(\alpha_1, \alpha_2)$, contradicting the assumption that $(t, u) \in \frob{\alpha_1, \alpha_2}$.

Therefore, $u \geq a_2 + \frac{tb_2}{a_2}$. But $(t, u) \in \frob{\alpha_1, \alpha_2}$ implies that $(t', u) \in \frob{\alpha_1, \alpha_2}$ for all $t'$ such that $t < t' < a_1$. Therefore, $u \geq a_2 + \frac{t' b_2}{a_2}$ for all such $t'$. Therefore, $u \geq a_2 + \frac{a_1 b_2}{a_2}$.

To finish the proof in this case, it will suffice to show that $[a_2, a_1) \times [a_2 + \frac{a_1 b_2}{a_2}, \infty) \subseteq \frob{\alpha_1, \alpha_2}$, and for that it will suffice to show that if $a_2 \leq f < a_1$ and $a_2 + \frac{a_1 b_2}{a_2} \leq g$ then $(f, g) \in MN(\alpha_1, \alpha_2)$. For such $(f, g)$, $(f, g) = \paren{\frac f {a_2}, \frac 1 {a_2} \paren{g - \frac {f b_2}{a_2}}} \alpha_2$, so $(f, g) \in MN(\alpha_1, \alpha_2)$ since $\frac f {a_2} \geq 1$ and $\frac 1{a_2}\paren{g - \frac {f b_2}{a_2}} \geq \frac 1 {a_2} \paren{a_2 + \frac{a_1 b_2}{a_2} - \frac{a_1 b_2}{a_2}} = 1$.

\tb{Case 4}: $0 < b_2 < a_2 < a_1$. Clearly $\frob{\alpha_1, \alpha_2} \subseteq (a_2, b_2) + [0, \infty)^2$. If $a_1 \leq f$ and $a_2 \leq g$ then $(f, g) = \paren{\frac f{a_1}, 0} \alpha_1 + \paren{0, \frac g{a_2}} \alpha_2$, so $(f, g) \in MN(\alpha_1, \alpha_2)$; therefore, $(a_1, a_2) + [0, \infty)^2 \subseteq \frob{\alpha_1, \alpha_2}$.

Next, we shall show that
\[ \frob{\alpha_1, \alpha_2} \cap \paren{[a_2, a_1) \times [b_2, \infty)} = [a_2, a_1) \times \bigg[a_2 + \frac{a_1 b_2}{a_2}, \infty\bigg) \text, \]
by an argument that will be familiar to anyone who has read the proofs in cases 2 and 3.

Suppose that $(t, u) \in \frob{\alpha_1, \alpha_2}$ and $a_2 \leq t < a_1$. Then for some $c_2, d_1, d_2 \in \set{0} \cup [1, \infty)$, $t = c_2 a_2$ and $u = a_1 d_1 + a_2 d_2 + \frac t{a_2} b_2$. If $u < a_2 + \frac{tb_2}{a_2}$, then $d_1 = d_2 = 0$ and $u = \frac{tb_2}{a_2}$. But then for all $u' \in \paren{u, \frac{t b_2}{a_2}}$, $(t, u') \notin MN(\alpha_1, \alpha_2)$, which contradicts the assumption that $(t, u) \in \frob{\alpha_1, \alpha_2}$.

Therefore, $u \geq a_2 + \frac{tb_2}{a_2}$. But then the fact that $(t', u) \in \frob{\alpha_1, \alpha_2}$ for all $t'$ satisfying $t < t' < a_1$ implies that $u \geq a_2 + \frac{t' b_2}{a_2}$ for all such $t'$. Therefore $u \geq a_2 + \frac{a_1 b_2}{a_2}$, which shows that
\[ \frob{\alpha_1, \alpha_2} \cap \paren{[a_2, a_1) \times [b_2, \infty) } \subseteq [a_2, a_1) \times \bigg[a_2 + \frac{a_1 b_2}{a_2}, \infty\bigg) \text. \]

On the other hand, if $a_2 \leq f < a_1$ and $g \geq a_2 + \frac{a_1b_2}{a_2}$, then $(f, g) = \paren{\frac{f}{a_2}, \frac{1}{a_2}\paren{g - \frac{fb_2}{a_2}}} \alpha_2$. 
From $\frac{f}{a_2} \geq 0$ and 
\[\frac{1}{a_2} \paren{g - \frac{f b_2}{a_2}} 
\geq \frac 1 {a_2}\paren{a_2 + \frac{a_1 b_2}{a_2} - \frac{a_1 b_2}{a_2}} = 1,\] 
we conclude that $(f, g) \in MN(\alpha_1, \alpha_2)$. Thus \[\frob{\alpha_1, \alpha_2} \cap \paren{[a_2, a_1) \times [b_2, \infty) } = [a_2, a_1) \times \bigg[a_2 + \frac{a_1 b_2}{a_2}, \infty\bigg).\]

To finish Case 4, we will determine $\frob{\alpha_1, \alpha_2} \cap \paren{[a_1, \infty) \times [b_2, a_2)}$. Suppose that $(t, u) \in \frob{\alpha_1, \alpha_2}$, with $t \geq a_1$ and $b_2 \leq u < a_2$. Then for some $c_1, c_2 \in \set{0} \cup [1, \infty)$, $t = a_1 c_1 + a_2 c_2$ and $u = b_2 c_2$. Then $t = a_1 c_1 + \frac{a_2 u}{b_2}$. If $t < a_1 + \frac{u a_2}{b_2}$, then $c_1 = 0$ and $t = \frac{u a_2}{b_2}$.

Then for all $t'$ such that $t < t' < a_1 + \frac{u a_2}{b_2}$, $(t', u) \notin MN(\alpha_1, \alpha_2)$, which contradicts the assumption that $(t, u) \in \frob{\alpha_1, \alpha_2}$. Therefore, $t \geq a_1 + \frac{u a_2}{ b_2}$. But then $(t, u) \in \frob{\alpha_1, \alpha_2}$ implies that $(t, u') \in \frob{\alpha_1, \alpha_2}$ for all $u'$ satisfying $u < u' < a_2$. Therefore, $t \geq a_1 + \frac{u' a_2}{b_2}$ for each such $u'$. Therefore, $t \geq a_1 + \frac{a_2^2}{b_2}$. Thus $\frob{\alpha_1, \alpha_2} \cap \paren{[a_1, \infty) \times [b_2, a_2)} \subseteq [a_1 + \frac{a_2^2}{b_2}, \infty) \times [b_2, a_2)$. On the other hand, suppose that $a_1 + \frac{a_2^2}{b_2} \leq f$ and $b_2 \leq g < a_2$. Then $(f, g) = \paren{\frac{1}{a_1}\paren{f - \frac{g a_2}{b_2}}, 0} \alpha_1 + \paren{\frac{g}{b_2}, 0} \alpha_2$. Thus, since $\frac{g}{b_2} \geq 1$ and $\frac{1}{a_1} \paren{f - \frac{g a_2}{b_2}} \geq \frac{1}{a_1} \paren{a_1 + \frac{a_2^2}{b_2} - \frac{a_2^2}{b_2}} = 1$, $(f, g) \in MN(\alpha_1, \alpha_2)$. Thus $\frob{\alpha_1, \alpha_2} \cap [a_1, \infty) \times [b_2, a_2) = [a_1 + \frac{a_2^2}{b_2}, \infty) \times [b_2, a_2)$. This, together with previous results, proves the claim in Case 4.
\end{proof}

We now shift focus to $Q = \Z$. Recall that $\cond{a_1, \dots, a_n}$ is well-defined for coprime positive integers $a_1, \dots, a_n$, with respect to the classical Frobenius template $(\N, \N, \N)$.

{\proposition For $n\geq 2$, let $(\alpha_1, \dots, \alpha_n)$ be a list of $2$-tuples $\alpha_i = (a_i,b_i) \in \Z^+ \times \N$ (for $i = 1, \dots, n$) satisfying $\gcd(a_1, \dots, a_n) = 1$ and $b_1 = 0$. Then \[ \paren{ \cond{a_1, \dots, a_n}, \cond{a_1, \dots, a_n} +  (a_1 - 1) \sum_{i = 2}^n b_i } + \N^2 \subseteq \frob{\alpha_1, \dots, \alpha_n}.\]}

\begin{proof}
Let $t = \cond{a_1, \dots, a_n}$ and $u = \cond{a_1, \dots, a_n} +  (a_1 - 1) (\sum_{j = 2}^n b_j)$. It's sufficient to show that $(t, u) \in \frob{\alpha_1, \dots, \alpha_n}$, i.e.\\ $(t, u) + \N^2 \subseteq MN(\alpha_1, \dots, \alpha_n)$. Pick an arbitrary $(f,g) \in (t,u) + \N^2$, i.e. choose integers $f \geq \cond{a_1, \dots, a_n}$ and $g \geq \cond{a_1, \dots, a_n} + \sum_{j = 2}^n b_j(a_1 - 1)$. To complete the proof, we will show that 
\[ (f,g) \in MN(\alpha_1, \dots, \alpha_n) = \sett{ \paren{\sum_{i=1}^n a_i c_i, \sum_{i=1}^n a_i d_i + \sum_{i=2}^n b_i c_i } : c_i, d_i \in \N}. \]
Since $f \geq \cond{a_1, \dots, a_n}$, we can write $f = \sum_{i = 1}^n a_i c_i$ for some nonnegative integers $c_1, \dots, c_n$. For each $i \in \set{2, \dots, n}$, use Euclidean division to write $c_i = a_1 q_i + \tilde{c}_i$ for some new nonnegative integers $q_i$ and $\tilde{c}_i \leq (a_1 - 1)$, then set $\tilde{c}_1 = c_1 + \sum_{i=2}^n a_i q_i$ so that $\sum_{i = 1}^n a_i c_i = f = \sum_{i = 1}^n a_i \tilde{c}_i$. Can we find nonnegative $d_1, \dots, d_n$ such that $g = \sum_{i=1}^n a_i d_i + \sum_{i=2}^n b_i \tilde{c}_i$? Well, $g - \sum_{i=2}^n b_i \tilde{c}_i \geq g - \sum_{i=2}^n b_i (a_1 - 1) \geq \cond{a_1, \dots, a_n}$, so yes. Therefore, $(f,g) \in MN(\alpha_1, \dots, \alpha_n)$.
\end{proof}

{\corollary For $n\geq 2$, let $(\alpha_1, \dots, \alpha_n)$ be a list of $2$-tuples $\alpha_i = (a_i,b_i) \in \Z^+ \times \N$ (for $i = 1, \dots, n$) satisfying $\gcd(a_1, \dots, a_n) = 1$. $\frob{\alpha_1, \dots, \alpha_n}$ is nonempty if and only if at least one $b_i = 0$.}

\begin{proof}
Combine propositions 4.1 and 4.4.
\end{proof}

We have seen results like this before, such as Theorem 2.1 of Section 2 and the paragraph preceding the proof of Proposition 4.3. Combining this corollary with the next result completely solves the Frobenius problem in the case $n = 2$ and $Q = \Z$.

{\proposition Suppose that $a_1, a_2 \in \Z^+, b \in \N$, $a_1$ and $a_2$ are coprime, and $\alpha_1 = (a_1, 0), \alpha_2 = (a_2, b)$. Then \[\frob{\alpha_1, \alpha_2} = \paren{ \cond{a_1, a_2}, \cond{a_1, a_2} + b(a_1 - 1) } + \N^2.\]}

\vspace{-20pt}
\begin{proof}
Let $(t, u) \in \N \times \N$. By Proposition 4.4, $t \geq \cond{a_1, a_2}$ and $u \geq \cond{a_1, a_2} + b( a_1 - 1)$ implies that $(t, u) \in \frob{\alpha_1, \alpha_2}$. For the converse, we will prove the contrapositive. Suppose that $t < \cond{a_1, a_2}$. Then $\cond{a_1, a_2} - 1 \geq t$ is not in $\set{a_1 c_1  + a_2 c_2 : c_1, c_2 \in \N}$, so $(\cond{a_1, a_2} - 1, u)$ is not in $\{ ( a_1 c_1 + a_2 c_2, a_1 d_1 + a_2 d_2 + b c_2) : c_1, c_2, d_1, d_2 \in \N \} = MN(\alpha_1, \alpha_2)$, yet $(\cond{a_1, a_2} - 1,u) \in (t,u) + \N^2$. Hence $(t,u) \notin \frob{\alpha_1, \alpha_2}$.

Now assume that $u < \cond{a_1, a_2} + b(a_1 - 1)$. Set $f = a_1 c_1 + a_2 (a_1 - 1)$ for some nonnegative $c_1$ large enough such that $f \geq t$. Consider any alternative expression of $f$ as $f = a_1 \tilde{c}_1 + a_2 \tilde{c}_2$ using nonnegative coefficients $\tilde{c}_1, \tilde{c}_2$. Recall that $f = a_1 c_1 + a_2 (a_1 - 1)$, so $a_1(c_1 - \tilde{c}_1) = a_2(\tilde{c}_2 - (a_1 - 1))$. Then $a_1 | (\tilde{c}_2 - (a_1 - 1))$ since $a_1$ and $a_2$ are coprime, so $\tilde{c}_2 \equiv a_1 - 1 \pmod{a_1}$. And $\tilde{c}_2$ is nonnegative, so $\tilde{c}_2 \geq (a_1 - 1)$, so $\tilde{c}_2 = (a_1 - 1) + a_1 k$ for some nonnegative integer $k$. Set $g = \cond{a_1, a_2} + b(a_1 - 1) - 1$ so that $g \geq u$; hence $(f,g) \in (t,u) + \N^2$. 
Suppose that $(f,g) \in MN(\alpha_1, \alpha_2)$; there are nonnegative coefficients $d_1$ and $d_2$ such that $g = a_1d_1 + a_2d_2 + b \tilde{c}_2 = a_1 d_1 + a_2 d_2 + b ((a_1 - 1) + a_1 k)$.  Consequently, $g - b(a_1 - 1) = a_1(d_1 + bk) + a_2d_2$ for nonnegative integers $(d_1 + b k)$ and $d_2$, so $g - b(a_1 - 1) \in MN(a_1, a_2)$. But this is impossible, because $g - b(a_1 - 1) = \cond{a_1,a_2} - 1 \notin MN(a_1, a_2)$, so our supposition must be false; $(f, g) \notin MN(\alpha_1, \alpha_2)$. Therefore, $(t, u) \notin \frob{\alpha_1, \alpha_2}$.
\end{proof}

Proposition 4.5 shows that when $n = 2$, the set inclusion in the conclusion of Proposition 4.4 is an equality. However, for lists of length $> 2$, there are counterexamples to the reverse set inclusion of Proposition 4.4, so Proposition 4.5 does not generalize to longer lists of tuples. The following is a counterexample: Let $\alpha_1=(3,0)$, $\alpha_2=(5,2)$, and $\alpha_3=(7,4)$. Then it can be shown that $(5,16) \in \frob{\alpha_1, \alpha_2, \alpha_3}+\N^2$, but $(5,16) \notin \{\cond{a_1, \dots, a_n} +  (a_1 - 1) \paren{\sum_{i = 2}^n b_i}\} + \N^2$. In fact, it can be shown that $\frob{\alpha_1, \alpha_2, \alpha_3}= (5,9) + \N^2$. Furthermore, for $\beta_1 = (3, 0)$, $\beta_2 = (5, 1)$, and $\beta_3 = (7, 4)$ we have
\[ \frob{\beta_1, \beta_2, \beta_3} = \paren{(5, 9) + \N^2} \cup \paren{(8, 7) + \N^2}  \text, \]
so the Frobenius set might even be a union of two sets, each of the form $(a, b) + \N^2$, neither contained in the other.
This situation resembles that of the classical template, in the sense that the Frobenius problem is completely solved for lists of length $2$, but not for lists of length $> 2$. 

\section{A more general template}

Here we broaden our horizons and pass from the templates $(A, C, U)$ where $A$ and $C$ are thought of as subsets of the same overlying ring, to templates $(A', C', U')$ where $A'$ is a monoid and $C'$ is a set of functions acting on $A'$. The first kind of template, i.e. the only kind hitherto discussed in this paper, can be considered a special case of the second; furthermore, templates of the second kind cannot in general be interpreted as examples of the first kind. The different entries in the new kind of template have the same roles as the corresponding entries in the original kind of template.

For the sake of presentation, we will showcase a certain example of this new kind of template and leave the precise definitions to the reader. Let $A = U = \nat \times \nat \times \nat,$ and let $C$ be the set of upper triangular matrices in $M_3(\nat)$.
In the ring $\zz^3$ with coordinate addition and multiplication, $A = U = \nat^3$ is a monoid. However, $C$ is not contained in $\zz^3,$ so this template is different from those considered previously. Note that $C$ contains an isomorphic copy of $\nat^3$ via the semiring embedding given by
$$
\begin{bmatrix} 
    x \\
    y \\
    z \\
\end{bmatrix} 
\mapsto
\begin{bmatrix} 
    x & 0 & 0 \\
    0 & y & 0 \\
    0 & 0 & z \\
\end{bmatrix}.
$$

Consider a general pair of tuples $(a, b, c)^T$ and $(d, e, f)^T \in A$. A general member of 
$$
    MN\left(\begin{bmatrix}
    a\\ b\\ c\\
    \end{bmatrix},
    \begin{bmatrix}
    d\\e\\f\\
    \end{bmatrix}\right)
$$
has the form
$$
    \begin{bmatrix}
    u & v & w \\ 
    0 & x & y \\
    0 & 0 & z \\ 
    \end{bmatrix}
    \begin{bmatrix} 
    a\\ b\\ c\\
    \end{bmatrix} 
    +
    \begin{bmatrix}
    u' & v' & w' \\
    0 & x' & y' \\
    0 & 0 & z' \\ 
    \end{bmatrix}
    \begin{bmatrix}
    d \\ e \\ f \\ 
    \end{bmatrix} \text,
$$
with all entries coming from $\N$. As expected, 
$\frob{(a, b, c)^T, (d, e, f)^T}$ is defined to be
\[
    \sett{\mb{w}\in MN\left(
    \begin{bmatrix}
    a\\ b\\ c\\
    \end{bmatrix},
    \begin{bmatrix}
    d\\e\\f\\
    \end{bmatrix}\right):\mb{w}+\N^3
    \stin MN\left(
    \begin{bmatrix}
    a\\ b\\ c\\
    \end{bmatrix},
    \begin{bmatrix}
    d\\e\\f\\
    \end{bmatrix}\right)} \text.
\]
Hence, by our understanding of the classical Frobenius problem, we see that 
$$
    \frob{\begin{bmatrix}
    a \\ b\\ c\\
    \end{bmatrix},
    \begin{bmatrix}
    d \\ e\\ f\\
    \end{bmatrix}}
    =
    \begin{bmatrix}
    \chi(a,b,c,d,e,f)\\
    \chi(b,c,e,f) \\ 
    \chi(c,f) \\
    \end{bmatrix} + \nat^3,
$$
when nonempty, which is true if and only if $\gcd(c,f) = 1$.

From the case $k = 2$ it is straightforward to see what $\frob{\alpha_1,...,\alpha_k}$ is for arbitrary $k \in \zz^+$ and $\alpha_1,...,\alpha_k \in A$. It is not difficult to generalize these results to $m \times 1$ column vectors and $m \times m$ matrices for $m > 3$. Therefore these cases are no longer interesting, except for the connection between them and the classical Frobenius problem.
We can get more challenging problems by restricting the matrices in $C$. For instance, we could require the matrices to be symmetric, or upper triangular with constant diagonal. 

These examples point to a generalized Frobenius template $(A', C', U')$ in which $A'$ (or perhaps $A' \cup \set{0}$) and $U'$ are monoids in a ring $R$, and $C'$ is a monoid in the ring of endomorphisms of $(R,+)$ such that each $\varphi \in C$ maps $A$ into $U.$ 

\medskip

\bibliographystyle{plain}
\bibliography{references}

\end{document}